\documentclass[reqno]{mfatshort}
\usepackage[english]{babel}
\usepackage{amsfonts}
\usepackage{amsmath}
\usepackage{amssymb}
\usepackage{amsthm}
\usepackage{pb-diagram}
\usepackage{latexsym}
\newtheorem{teo}{Theorem}
\newtheorem{lem}{Lemma}
\newtheorem{df}{Definition}

\makeatletter
\newcommand\testshape{family=\f@family; series=\f@series; shape=\f@shape.}
\def\myemphInternal#1{\if n\f@shape%
\begingroup\itshape #1\endgroup\/%
\else\begingroup\bfseries #1\endgroup%
\fi}
\def\myemph{\futurelet\testchar\MaybeOptArgmyemph}
\def\MaybeOptArgmyemph{\ifx[\testchar \let\next\OptArgmyemph
                 \else \let\next\NoOptArgmyemph \fi \next}
\def\OptArgmyemph[#1]#2{\index{#1}\myemphInternal{#2}}
\def\NoOptArgmyemph#1{\myemphInternal{#1}}
\makeatother

\newcommand{\rr}{\mathbb{R}}

\usepackage{graphicx}

\begin{document}

\title[Topological equivalence to a projection]
	{Topological equivalence to a projection}

\author{\fbox{V. V. Sharko}}
\address{Institute of Mathematics  National Academy of Sciences of Ukraine, }
\email{sharko@imath.kiev.ua}

\author{Yu. Yu. Soroka}
\address{Taras Shevchenko National University of Kiev}
\email{soroka.yulya@ukr.net}
\subjclass[2000]{Primary ; Secondary }
\date{05/11/2014}
\keywords{topological equivalence, a linear function}

\begin{abstract}
We present a necessary and sufficient condition for the topological equivalence of a continuous function on a plane to a projection onto one of coordinates.
\end{abstract}

\maketitle

Let $M$ be a connected surface, i.e. $2$-dimensional manifold.
Two continuous functions $f, g:M \to \rr$ are called \myemph{topologically equivalent}, if there exist two homeomorphisms $h:M \to M$ and $k:\rr \to \rr$ such that $k \circ f=g \circ h$.

Classification of continuous functions $f:M \to \rr$ on surfaces up to topological equivalence was initiated in the works by M.~Morse~\cite{Morse:DMJ:1946}, \cite{Morse}, see also \cite{JenkinsMorse:AJM:1952, JenkinsMorse:AMS:1953, JenkinsMorse:UMP:1955}.
In recent years essential progress in the classification of such functions was made in~\cite{Oshemkov, Sharko2, Sharko3, Arnold, PolulyakhYurchuk}.


%
%
%
%
%

Let $f:\rr^2 \to \rr$ be a continuous function.
Assuming that $f$ ``has not critical points'' we present a necessary and sufficient condition for $f$ to be topologically equivalent to a linear function.
First we recall some definitions from W.~Kaplan~\cite{Kaplan}. 

\begin{df}{\rm\cite{Kaplan}}.
A \myemph{curve} in $\rr^2$ is a homeomorphic image of the open interval $(0,1)$.
Let $U\subset\rr^2$ be an open subset.
A \myemph{family of curves} in $U$ is a partition of $U$ whose elements are curves.

A family of curves $\Im$ in $U$ is called \myemph{regular} at a point $p\in\rr^2$, if there exists an open neighbourhood $U_p$ of $p$ and a homeomorphism $\varphi:(0,1)\times (0,1) \to U_p$ such that for every $y\in(0,1)$ the image $\varphi((0,1)\times y)$ is an intersection of $U_p$ with some curve from the family $\Im$.
Such a neighbourhood $U_p$ is called \myemph{$r$-neighbourhood} of $p$.
\end{df}

Thus the curves of regular family are ``locally parallel'', however they global behaviour can be more complicated, see Figure~\ref{ris:pic}.
We will consider continuous functions $f:\rr^2\to\rr$ whose level-sets are ``globally parallel''.

One of the basic examples is a projection $g:\rr^2\to\rr$ given by $g(x,y)=y$.
Its level sets are parallel lines $y=\mathrm{const}$, and in particular they constitute a regular family of curves.

On the other hand, consider the function $f(x,y)=\arctan(y-\mathrm{tg}^{2}(x))$, see Figure~\ref{ris:pic}.
It level sets are not connected, however, the partition into \myemph{connected components} of level sets of $f$ is also a regular family of curves.
\begin{figure}[ht]
\center{\includegraphics[height = 2.5cm]{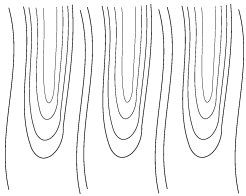}}
\caption{Level lines of $f(x,y)=\arctan( y-\mathrm{tg}^{2}(x))$}
\label{ris:pic}
\end{figure}

The following theorem shows that connectedness of level sets is a characteristic property of a projection.
%
%
\begin{teo}\label{th:main_theorem}
Let $f:\rr^2\to\rr$ be a continuous function and $\Im = \{ f^{-1}(a)\mid a\in\rr^2\}$ be the partition of $\rr$ by level sets of $f$.
Suppose the following two conditions hold.
\begin{enumerate}
\item 
For each $p\in f(\rr^2)$ belonging to the image of $f$, the corresponding level set $f^{-1}(a)$ is a \myemph{curve}.
In particular, it is path connected.

\item The family of curves $\Im$ is regular.
\end{enumerate}
Then the image $f\,(\rr^{2})$ is an open interval $(a, b)$, $a,b \in \rr\cup{\pm\infty}$, and there is a homeomorphism $\varphi: \rr \times (a,b) \to \rr^2$ such that $f \circ \varphi(x,y) = y$.
In other words, $f$ is topologically equivalent to a projection.
\end{teo}

The proof is based on the results of~\cite{Kaplan}.

Let $\Im$ be a regular family of curves in $\rr^2$.
Then by~\cite[Theorem 16]{Kaplan}, each curve $C$ of $\Im$ is a proper embedding of $\rr$, so it has an infinity as a sole limit point.
It follows from Jordan's theorem (applied to the sphere $S^2=\rr^2\cup\infty$) that each curve $C$ of $\Im$ divides the plane into two distinct regions, having $C$ as the common boundary.
This property allows to define the following relation for curves on $\Im$.
\begin{df}
Three distinct curves $K,\, L, \, C$ from $\Im$ are \myemph{in the relation $K|C|L$}, if $K$ and $L$ belong to distinct components of $\rr^2\setminus C$.
\end{df}

%
%
%
%

For an open subset $U\subset\rr^2$ let $\Im_{U}$ be the partition of $U$ by connected components of intersections of $U$ with curves from $\Im$.
Then $\Im_{U}$ is a family of curves in $U$.
Notice also that an intersection of $U$ with some curve from $\Im$ may have even countable many connected components.
\begin{df}
Let $p,q\in\rr^2$.
An \myemph{arc} $[p,q]$, i.e. homeomorphic image of $[0,1]$, connecting these points, will be called a \myemph{cross-section} relative to   $\Im$ if there exists an open set $U$ in $\rr^2$ containing $[p,q]$ and such that each curve of $\Im_{U}$ meets $[p,q]$ in $U$ at most once.
\end{df}

Evidently, for every $p\in\rr^2$, there is an arbitrary small $r$-neighbourhood $V$ of $p$ and a cross-section $[q,s] \subset V$ relative $\Im$ passing through $p$.

\begin{teo}\label{th:Kaplan}{\rm\cite{Kaplan}}
Let $K,L$ be two distinct curves from a regular family $\Im$.
Suppose two points $p\in L$ and $q\in K$ can be connected by a cross-section $[p,q]$ and let $S$ be the set of curves crossing $[p,q]$ except for $p$ and $q$.
Then $S$ form an open point set and the condition $K|C|L$ is equivalent to the condition that $C$ is contained in $S$.

Moreover, there is a homeomorphism $\varphi:\rr\times[0,1] \to K \cup S \cup L$ such that $K = \phi(\rr\times0)$, $L = \phi(\rr\times 1)$, and $\phi(\rr\times t)$ is a curve belonging to $\Im$ for all $t\in(0,1)$.
\end{teo}
\begin{proof}
First we need the following lemma.
\begin{lem}\label{lm:monot_cross_sect}
Let $[p,q]$ be a cross-section of $\Im$.
Then the restriction of $f$ to $[p,q]$ is strictly monotone.
In particular, $[p,q]$ intersects each curve in $\Im$ in at most one point.
\end{lem}
\begin{proof}
Suppose there exists a point $x\in[p,q]$ distinct from $p$ and $q$ and being a local extreme of $f|_{[p,q]}$.
Let $c=f(x)$.
As mentioned above the embedding $f^{-1}(c) \subset \rr^2$ is proper, therefore 
\begin{enumerate}
\item[(i)]
$f^{-1}(c)$ divides $\rr^2$ into two connected components, say $R_1$ and $R_2$ and
\item[(ii)]
there exists an $r$-neighbourhood $U$ of $x$ relatively to $\Im$ such that $U\cap f^{-1}(c)$ is a connected curve dividing $U$ into two components, say $U_1$ and $U_2$, such that $U_1\subset R_1$ and $U_2\subset R_2$.
\end{enumerate}
Not loosing generality, we can assume that $[p,q] \subset U$ so that $[p,q]\setminus\{x\}$ consists of two half-open arcs $[p,x) \subset U_1$ and $(x,q] \subset U_2$.
It follows that $x$ is an \myemph{isolated} local extreme of the restriction of $f|_{[p,q]}$, whence there exist $y\in[p,x)\subset R_1$ and $z\in(x,q] \subset R_2$ such that $f(y)=f(z)\not=f(c)$.
Thus $y, z \in f^{-1}(f(y)) \subset \rr^2\setminus f^{-1}(c) = R_1 \cup R_1$.
By (1) $f^{-1}(f(y))$ is connected, and so both $y$ and $z$ belong either to $R_1$ or to $R_2$.
This gives a contradiction, whence $x$ is not a local extreme of $f$.
\end{proof}

For $[c,d]\subset\rr$ denote $D_{c,d} = f^{-1}[c,d]$.
Then it follows from Lemma~\ref{lm:monot_cross_sect} and Theorem~\ref{th:Kaplan} that for each cross-section $[p,q]$ there exists a homeomorphism \[ \varphi: \rr \, \times \,[f(p), f(q)] \longrightarrow f^{-1}[f(p), f(q)] = D_{f(p),f(q)}\]
such that $f\circ\varphi(x,y) = y$ for all $(x,y)\in \rr \times [f(p), f(q)]$.
%

This also implies that the image $f(\rr^2)$ is an open and path connected subset of $\rr$, i.e. an open interval $(a,b)$, where $a$ and $b$ can be infinite.

Hence we can find a countable strictly increasing sequence $\{c_i\}_{i\in\mathbb{Z}} \subset \rr$ such that
$\lim\limits_{k\to-\infty}c_i = a$, $\lim\limits_{k\to+\infty}c_k = b$, and for each $k\in\mathbb{Z}$ a homeomorphism
\[ \varphi_{k}: \rr\, \times \,[c_{k};c_{k+1}] \longrightarrow f^{-1}[c_k,c_{k+1}] = D_{c_{k}, c_{k+1}}\]
satisfying $f \circ \varphi_{k}(x,y) = y$.   

%
%

Define a homeomorphism $\varphi: \rr \times (a,b) \to \rr^2$ as follows.
Set \[ \varphi(x,y) = \varphi_{0}(x,y), \qquad  (x,y)\in \rr\times [c_0,c_{1}].\]
Now if $\varphi$ is defined on $\rr \times [c_{k-1}, c_{k}]$ for some $k\geq1$, then extend it to $\rr\times [c_{k}, c_{k+1}]$ by
\[
\varphi(x,y) = \varphi_{k} (\varphi^{-1}_{k} \circ \varphi(x,c_k), y), 
\qquad (x,y) \in \rr\times[c_{k},c_{k+1}].
\]
Similarly, one can extend $\varphi$ to $\rr\times(a,c_0]$.
It easily follows that $\varphi$ is a homeomorphism satisfying $f\circ\varphi(x,y)=y$, $(x,y)\in\rr\times(a,b)$.
\end{proof}

\end{document}